\documentclass{amsart}
\newcommand{\BibTeX}{{\rm B\kern-.05em{\sc
          i\kern-.025emb}\kern-.08em\TeX}}
\usepackage{rotating}
\label{e:dispaa}
\label{e:dispau}
\label{e:dispav}
\label{e:dispaw}
\label{e:dispax}
\makeatother

\usepackage{amsmath}
\usepackage{mathrsfs}    %pacote para a letra caligráfica:
\usepackage{amssymb}
\usepackage{amsthm}
\usepackage{graphicx,color}
\usepackage{hyperref}
\usepackage{verbatim} %isto serve para poder colocar \begin{comment} \end{comment}
\usepackage[applemac]{inputenc}

\newtheorem{thm}{Theorem}[section]
\newtheorem{prop}[thm]{Proposition}

\newtheorem{lemma}[thm]{Lemma}

\newtheoremstyle{definition2}{\topsep}{\topsep}%
     {}%         Body font
     {}%         Indent amount (empty = no indent, \parindent = para indent)
     {\bfseries}% Thm head font
     {.}%        Punctuation after thm head
     {.5em}%     Space after thm head (\newline = linebreak)
     {\thmnumber{(#2)}\thmname{ #1}\thmnote{ #3}}%         Thm head spec

\theoremstyle{definition}
\newtheorem{rem}[thm]{Remark}

\def\ep{\varepsilon}
\def\vphi{\varphi}

\def\R{\mathbb{R}}
\def\K{\mathbb{K}}

\def\K{\mathbb{K}}

\def\hbar{\bar{h}}

\def\dist{{\rm dist}}

\def\leq{\leqslant}
\def\geq{\geqslant}

\definecolor{verde}{rgb}{0,0.35,0.1} %0.35, this way is darker
\definecolor{rosso}{rgb}{0.7,0,0}
\definecolor{blue}{rgb}{0,0,1}
\definecolor{viola}{rgb}{0.6,0,0.4}

\newcommand{\rst}[1]{\ensuremath{{\mathbin |}%
\raise-.5ex\hbox{$#1$}}} 

\begin{document}

\title[]{Increasing powers in a degenerate parabolic logistic equation}
\author{Jos\'e Francisco Rodrigues}
\address{Universidade de Lisboa/CMAF 
Av. Prof. Gama Pinto 2 
1649–003 Lisboa, Portugal}

\email{rodrigue@ptmat.fc.ul.pt }
\thanks{Both authors were supported by Funda\c c\~ao para a Ci\^encia e a Tecnologia (FCT), PEst OE/MAT/UI0209/2011.
}

\author{Hugo Tavares}
\address{Universidade de Lisboa/CMAF 
Av. Prof. Gama Pinto 2 
1649–003 Lisboa, Portugal}
\email{htavares@ptmat.fc.ul.pt}
\thanks{The second author was also supported by the FCT grant SFRH/BPD/69314/201.}

%    General info

\date{May 21, 2012}


\begin{abstract}
The purpose of this paper is to study the asymptotic behavior of the positive solutions of the problem
$$
\partial_t u-\Delta u=a u-b(x) u^p \text{ in } \Omega\times \R^+,\ u(0)=u_0,\ u(t)|_{\partial \Omega}=0
$$
as $p\to +\infty$, where $\Omega$ is a bounded domain and $b(x)$ is a nonnegative function. We deduce that the limiting configuration solves a parabolic obstacle problem, and afterwards we fully describe its long time behavior.
\end{abstract}

\maketitle

\section{Introduction}

In this paper we are interested in the study of the parabolic problem
\begin{equation} \label{eq:parabolic}
\left\{\begin{array}{rcll}
\partial_t u-\Delta u&=&a u-b(x) u^p & \text{ in } Q:=\Omega \times (0,+\infty)\\
u &=& 0 &\text{ on } \partial\Omega \times (0,+\infty)\\
u(0)&=& u_0 & \text{ in }\Omega
\end{array}
\right.
\end{equation}
where $a>0$, $p>1$, $b\in L^\infty(\Omega)$ is a nonnegative function and $\Omega$ is a bounded domain with smooth boundary. Such system arises in population dynamics, where $u$ denotes the population density of a given specie, subject to a logistic-type law. 

It is well known that under these assumptions and for very general $u_0$'s, \eqref{eq:parabolic} admits a unique global positive solution $u_p=u_p(x,t)$. In fact, in order to deduce the existence result, one can make the change of variables $v=e^{-at}u$, and deduce that $v$ satisfies $\partial_t v-\Delta v+b(x)e^{pat}v^p=0$. As $v\mapsto b(x)e^{pat}|v|^{p-1}v$ is monotone nondecreasing, the theory of monotone operators (cf. \cite{Lions, Zheng}) immediately provides existence of solution for the problem in $v$, and hence also for \eqref{eq:parabolic}.

One of our main interests is the study of the solution $u_p$ as $p\to +\infty$. As we will see, in the limit we will obtain a parabolic obstacle problem, and afterwards we fully describe its asymptotic limit as $t\to +\infty$.

This study is mainly inspired by the works of Dancer et al \cite{DD,DDM1,DDM2}, where the stationary version of \eqref{eq:parabolic} is addressed. Let us describe in detail their results. Consider the elliptic problem
\begin{equation}\label{eq:elliptic}
-\Delta u=a u-b(x) u^p,\qquad u\in H^1_0(\Omega)
\end{equation} 
and, for each domain $\omega\subseteq \R^N$, denote by $\lambda_1(\omega)$ the first eigenvalue of $-\Delta$ in $H^1_0(\omega)$. Assuming $b\in C(\overline \Omega)$, the study is divided in two cases: the so called \emph{nondegenerate} case (where $\min_{\overline \Omega} b(x)>0$) and the \emph{degenerate} one (where $\Omega_0:=\text{int} \{x\in \Omega:\ b(x)=0\}\neq\emptyset$ and has smooth boundary).

In the nondegenerate case, it is standard to check (see for instance \cite[Lemma 3.1 \& Theorem 3.5]{FKLM}) that \eqref{eq:elliptic} has a positive solution if and only if $a>\lambda_1(\Omega)$. For each $a>\lambda_1(\Omega)$ fixed, then in \cite{DDM1} it is shown that $u_p\to w$ in $C^1(\overline \Omega)$ as $p\to +\infty$, where $w$ is the unique solution of the obstacle-type problem
\begin{equation}\label{eq:obstacle_elliptic_0}
-\Delta w=a w \chi_{\{w<1\}},\quad w>0,\qquad w|_{\partial\Omega}=0,\ \|w\|_\infty=1.
\end{equation}
It is observed in \cite{DD} that $u$ is also the unique positive solution of the variational inequality
\begin{equation}\label{eq:obstacle_elliptic_1}
w\in \K:\qquad \int_\Omega \nabla w\cdot \nabla (v-w)\, dx\geq \int_\Omega a w(v-w)\, dx,\quad \forall v\in \K,
\end{equation}
where 
\begin{equation*}
\K=\{w\in H^1_0(\Omega):\ w\leq 1 \text{ a.e. in }\Omega\}.
\end{equation*}
In the degenerate case, on the other hand, problem \eqref{eq:elliptic} has a positive solution if and only if $a\in (\lambda_1(\Omega),\lambda_1(\Omega_0))$. For such $a$'s, assuming that $\Omega_0\Subset \Omega$, if we combine the results in \cite{DDM1,DDM2}, we see that $u_p\to w$ in $L^q(\Omega)$ for every $q\geq 1$, where now $w$ is the unique nontrivial nonnegative solution of
\begin{equation}\label{eq:obstacle_elliptic_2}
v\in \K_0:\qquad \int_\Omega \nabla w\cdot \nabla (v-w)\, dx\geq \int_\Omega a w(v-w)\, dx, \quad \forall v\in \K_0,
\end{equation}
with
\begin{equation*}
\K_0=\{w\in H^1_0(\Omega):\ w\leq 1 \text{ a.e. in }\Omega\setminus \Omega_0\}.
\end{equation*}
The uniqueness result was the subject of the paper \cite{DDM2}. Therefore, whenever $b(x)\neq 0$, the term $b(x)u^p$ strongly penalizes the points where $u_p>1$, forcing the limiting solution to be bellow the obstacle $1$ at such points.

Our first aim is to extend these conclusions for the parabolic case \eqref{eq:parabolic}. While doing this, our concern was also to relax some of the assumptions considered in previous papers, namely the continuity of $b$ as well as the condition of $\Omega_0$ being in the interior of $\Omega$. In view of that, consider the following conditions for $b$:

\begin{itemize}
\item[(b1)] $b\in L^\infty(\Omega)$;
\item[(b2)] there exists $\Omega_0$, an open domain with smooth boundary, such that
$$
b(x)=0\text{ a.e. on }\Omega_0, \text{ and }
$$
$$
\forall\ \Omega'\Subset \Omega\setminus \Omega_0 \text{ open }\exists\ \underline{b}>0 \text{ such that }b(x)\geq \underline{b} \text{ a.e. in }\Omega'.
$$
\end{itemize}
Observe that in (b2) $\Omega_0=\emptyset$ is allowed, and $\overline \Omega_0$ may intersect $\partial \Omega$. Continuous functions with regular nodal sets or characteristic functions of open smooth domains are typical examples of functions satisfying (b1)-(b2). As for the initial data, we consider:

\begin{itemize}
\item[(H1)] $u_0\in H^1_0(\Omega)\cap L^\infty(\Omega)$;
\item[(H2)] $0\leq u_0\leq 1$ a.e. in $\Omega \setminus \Omega_0$.
\end{itemize}

Our first main result is the following.

\begin{thm}\label{thm:3}
Assume that $b$ satisfies (b1)-(b2) and $u_0$ satisfies (H1)-(H2). Then there exists a function $u$ such that, given $T>0$, $u\in L^\infty(0,T; H^1_0(\Omega))\cap H^1(0,T;L^2(\Omega))$ and
\[
\begin{split}
&u_p\to u \qquad \text{strongly in }L^2(0,T;H^1_0(\Omega));\\
&\partial_t u_p\rightharpoonup \partial_t u \qquad \text{ weakly in } L^2(Q_T).
\end{split}
\]
Moreover $u$ is the unique solution of the following problem:
\begin{quote}
$\text{for a.e. }t>0,\ u(t)\in \K_0:$
\begin{equation}\label{eq:limit_problem_2}
\int_\Omega \partial_t u(t) (v-u(t))\, dx+\int_\Omega \nabla u(t)\cdot \nabla (v-u(t))\, dx\geq \int_\Omega a u(t)(v-u(t))\, dx
\end{equation}
for every $v\in \K_0$, with the initial condition $u(0)=u_0$.
\end{quote}
\end{thm}
 
 Next we turn to the long time behavior of the solution of \eqref{eq:limit_problem_2}.
 
 \begin{thm}\label{thm:4}
Suppose the $b$ satisfies (b1)-(b2) and take $u_0$ verifying (H1)-(H2). Fix $a\in (\lambda_1(\Omega),\lambda_1(\Omega_0))$. Let $u$ be the unique positive solution of \eqref{eq:limit_problem_2} and take $w$ the unique nontrivial nonnegative solution of \eqref{eq:obstacle_elliptic_2}. Then $\|w\|_\infty=1$ and
$$
u(t)\to w \qquad \text{ strongly in } H^1_0(\Omega), \text{ as }t\to +\infty.
$$
Moreover if $a<\lambda_1(\Omega)$ then $\|u(t)\|_{H^1_0(\Omega)}\to 0$, and if $a\geq \lambda_1(\Omega_0)$ then both $\|u(t)\|_\infty$ and $\|u(t)\|_{H^1_0(\Omega)}$ go to $+\infty$ as $t\to +\infty$.
\end{thm}

We remark that in the case $\Omega_0=\emptyset$ we set $\lambda_1(\Omega_0):=+\infty$, and $a\geq \lambda_1(\Omega_0)$ is an empty condition. The case $a= \lambda_1(\Omega)$ is the subject of Remark \ref{remark:a_leq_lambda_1}.

Under some stronger regularity assumptions on $b$, $u_0$ and $\Omega_0$, it is known (cf. \cite[Theorem 3.7]{FKLM} or \cite[Theorem 2.2]{DuGuo}) that $u_p(t,x)$ converges to the unique positive solution of \eqref{eq:elliptic} whenever $a\in (\lambda_1(\Omega),\lambda_1(\Omega_0))$. Hence in this situation, if we combine all this information together with the results  obtained in this paper, then we can conclude that the following diagram commutes:

\vspace{0.1cm}
\begin{center}
\begin{tabular}{ccc}
$u_p(x,t)$ positive solution of &  & 						 \\[-2.5ex]
   $\partial_t u-\Delta u=a u-b(x)u^p$			    & \raisebox{1.5ex}{$\stackrel{p\to +\infty}{  \xrightarrow{\hspace*{2.5cm}}}$} &    \raisebox{1.5ex}{  $u(x,t)$ solution of \eqref{eq:limit_problem_2}}\\
 \begin{sideways}$\stackrel{t\to +\infty}{\xleftarrow{\hspace*{1.5cm}}}$\end{sideways}	&      &   \begin{sideways}$\stackrel{t\to +\infty}{\xleftarrow{\hspace*{1.5cm}}}$\end{sideways} \\
 $u_p\in H^1_0(\Omega)$ positive solution of  &	&   $u\in H^1_0(\Omega)$ nontrivial  \\[-2.5ex]
 $-\Delta u=au-b(x)u^p$   &   \raisebox{1.5ex}{$\stackrel{p\to +\infty}{  \xrightarrow{\hspace*{2.5cm}}}$} &nonnegative solution of \eqref{eq:obstacle_elliptic_2}\\
\end{tabular}
\end{center}
\vspace{0.2cm}

The proof of Theorem \ref{thm:3} uses a different approach with respect to the works by Dancer et al.. While in \cite{DDM1} the authors use fine properties of functions in Sobolev spaces, here we follow some of the ideas present in the works \cite{boccardo,aglio}, and show that a uniform bound on the quantity 
$$
\iint_{Q_T} b(x)u_p^{p+1}\, dx dt \qquad (\text{for each } T>0),
$$
implies that $u(t)\in \K_0$ for a.e. $t>0$ (see the key Lemma \ref{lemma:key_lemma} ahead). As for the proof of Theorem \ref{thm:4}, the most difficult part is to show that when $a\in (\lambda_1(\Omega),\lambda_1(\Omega_0))$, $u_p(x,t)$ does not go to the trivial solution of \eqref{eq:obstacle_elliptic_2}. The key point here is to construct a subsolution of \eqref{eq:parabolic} independent of $p$. It turns out that to do this one needs to get a more complete understanding of the nondegenerate case, and to have a stronger convergence of $u_p$ to $u$ as $p\to +\infty$. So we dedicate a part of this paper to the study of this case. To state the results, let us start by defining for each $0<t_1<t_2$ and $Q_{t_1,t_2}:=\Omega\times (t_1,t_2)$ the spaces $C_\alpha^{1,0}(\overline Q_{t_1,t_2})$ and $W^{2,1}_q(Q_{t_1,t_2})$. For $q\geq 1$, the space $W^{2,1}_{q}(Q_{t_1,t_2})$ is the set of elements in $L^q(Q_{t_1,t_2})$ with partial derivatives $\partial_t u$, $D_xu$, $D_{x}^2 u$ in $L^q(Q_{t_1,t_2})$. It is a Banach space equipped with the norm
$$
\|u\|_{2,1;q,Q_{t_1,t_2}}=\|u\|_{L^q(Q_{t_1,t_2})}+\|D_x u\|_{L^q(Q_{t_1,t_2})}+\|D^2_{x}u\|_{L^q(Q_{t_1,t_2})}+\|\partial_t u\|_{L^q(Q_{t_1,t_2})}.
$$
For each $\alpha\in (0,1)$, $C^{1,0}_\alpha(\overline Q_{t_1,t_2})$ is the space of H\"older functions $u$ in $\overline Q_{t_1,t_2}$ with exponent $\alpha$ in the $x$--variable and $\alpha/2$ in the $t$--variable, with $D_x u$ satisfying the same property. More precisely, defining the H\"older semi-norm
\[
[u]_{\alpha,Q_{t_1,t_2}}:=\sup\left\{ \frac{|u(x,t)-u(x',t')|}{|x-x'|^\alpha+|t-t'|^{\alpha/2}},\quad x,x'\in \overline \Omega,\ t,t'\in [t_1,t_2],\ (x,t)\neq (x',t')\right\},
\]
we have that
\[
\begin{split}C^{1,0}_\alpha(\overline Q_{t_1,t_2}):=\Bigl\{u: \ \|u\|_{C^{1,0}_\alpha(\overline Q_{t_1,t_2})}:=& \|u\|_{L^\infty(Q_{t_1,t_2})}+\|D_x u\|_{L^\infty(Q_{t_1,t_2})}\\
			&+[u]_{\alpha,Q_{t_1,t_2}}+[D_x u]_{\alpha,Q_{t_1,t_2}}<+\infty \Bigr\}.  \end{split}
\]
Recall that we have the following embedding for every $0\leq t_1<t_2$ (see \cite[Lemmas II.3.3, II.3.4]{lady}):
\begin{equation}\label{eq:embedding}
W^{2,1}_q(Q_{t_1,t_2})\hookrightarrow C^{1,0}_\alpha(\overline Q_{t_1,t_2}),\qquad \forall 0\leq \alpha <1-\frac{N+2}{q}.
\end{equation}

In the nondegenerate case, we have the following result.

\begin{thm}\label{thm:1}
Suppose that $b$ satisfies (b1) and
\begin{itemize}
\item[(b2')] there exists $b_0>0$ such that $b(x)\geq b_0$ for a.e. $x\in \Omega$;
\end{itemize}
and that $u_0$ satisfies (H1) and $0\leq u_0\leq 1$ for a.e. $x\in \Omega$. Then, in addition to the conclusions of Theorem \ref{thm:3}, we have that
\begin{equation*}
u_p\to u \qquad \text{strongly in } C^{1,0}_\alpha(\overline Q_{t_1,t_2}), \text{ weakly in } W^{2,1}_q(Q_{t_1,t_2}) \text{ as }p\to +\infty,
\end{equation*}
for every $\alpha\in (0,1)$, $q\geq 1$ and $0<t_1<t_2$. Moreover, $u$ is the unique solution of
\begin{equation}\label{eq:limit_problem_1}
\partial_t u-\Delta u=au \chi_{\{u<1\}} \text{ in }Q,\quad u(0)=u_0,\qquad \|u\|_\infty\leq 1.
\end{equation}
\end{thm}

In this case, as $t\to +\infty$, we also obtain a convergence result for the coincidence sets $\{u(x,t)=1\}$.

\begin{thm}\label{thm:2}
Suppose that $b$ satisfies (b1)-(b2') and take $u_0$ satisfying (H1) and $0\leq u_0\leq 1$ for a.e. $x\in \Omega$. Fix $a>\lambda_1(\Omega)$ and let $u$ be the unique solution of \eqref{eq:limit_problem_1} and take $w$ the unique solution of \eqref{eq:obstacle_elliptic_0}. Then, as $t\to +\infty$,
$$
u(t)\to w \qquad \text{ strongly in } H^1_0(\Omega)\cap H^2(\Omega),
$$
and
\begin{equation}\label{eq:charact_1}
\chi_{\{u=1\}}(t)\to \chi_{\{w=1\}} \qquad \text{ strongly in } L^q(\Omega)\ \forall q\geq1.
\end{equation}
\end{thm}

The structure of this paper is as follows: in Section \ref{sec:section b(x)=0} we prove Theorem \ref{thm:3}, while in Section \ref{sec:section b(x)>0} Theorem \ref{thm:1} is treated. Finally, in Section \ref{sec:time} we show use the strong convergence up to the boundary of $\Omega$ obtained in the latter theorem  to prove Theorem \ref{thm:2}, and afterwards we use it combined with a subsolution argument to prove Theorem \ref{thm:4}.

We end this introduction by pointing out some other works concerning this type of asymptotic limit. The generalization of \cite{DDM1} for the $p$-Laplacian case was performed in \cite{GuoMa}. In \cite{boccardo,aglio}, elliptic problems of type
$$
-\Delta u+f(x,u)|f(x,u)|^p=g(x)
$$ 
are treated, while in the works by Grossi et al \cite{G,GN}, and Bonheure and Serra \cite{BS}, the authors dealt with the asymptotics study as $p\to +\infty$ of problems of type
$$
-\Delta u+V(|x|)u=u^p
$$
in a ball or annulus both with Neumann and Dirichlet boundary conditions. 

%%%%
%%%%SECTION 2
\section{The general case: Proof of Theorem \ref{thm:3}}\label{sec:section b(x)=0}

To make the presentation more structured, we split our proof in several lemmas. We start by showing a very simple comparison principle, which is an easy consequence of the monotonicity of the operator $u\mapsto |u|^{p-1}u$.

\begin{lemma}\label{lemma_subsuper}
Suppose that $u$ is a solution of \eqref{eq:parabolic} and take $v$ a \emph{supersolution} satisfying
\begin{equation*}
\left\{
\begin{array}{l}
\partial_t v-\Delta v\geq a v-b(x) v^p  \text{ in } Q_T\\
v(0) = v_0,\quad  v(t)|_{\partial \Omega}=0
\end{array}
\right. \qquad \text{ with } u_0\leq v_0.
\end{equation*}
Then $u(x,t)\leq v(x,t)$ a.e.. On the other hand if $v$ is a subsolution satisfying 
\begin{equation*}
\left\{
\begin{array}{l}
\partial_t v-\Delta v\leq a v-b(x) v^p  \text{ in } Q_T\\
v(0) = v_0,\quad  v(t)|_{\partial \Omega}=0
\end{array}
\right. \qquad \text{ with } v_0\leq u_0,
\end{equation*}
then $v(x,t)\leq u(x,t)$.
\end{lemma}
\begin{proof}
The proof is quite standard, but we include it here only for the sake of completeness. In the case where $v$ is a supersolution, we have
$$
\partial_t (u-v)-\Delta (u-v)+b(x)(u^p-v^p)\leq a(u-v).
$$
Multiplying this by $(u(t)-v(t))^+$, we obtain
\begin{multline*}
\frac{1}{2}\frac{d}{dt}\int_\Omega [(u(t)-v(t))^+]^2\, dx+\int_\Omega |\nabla (u(t)-v(t))^+|^2\, dx\\
 +\int_\Omega b(x)(u^p(t)-v^p(t))(u(t)-v(t))^+\, dx \leq a\int_\Omega [(u(t)-v(t))^+]^2\, dx
\end{multline*}
As $b(x)(u^p-v^p)(u-v)^+\geq 0$, we have
$$
\frac{d}{dt}\int_\Omega [(u(t)-v(t))^+]^2\, dx\leq 2a\int_\Omega [(u(t)-v(t))^+]^2\, dx,
$$
whence
$$
\int_\Omega [(u(t)-v(t))^+]^2\, dx\leq e^{2at}\int_\Omega [(u_0-v_0)^+]^2\, dx=0. 
$$ The proof of the result for the subsolution case is analogous.
\end{proof}

Next we show some uniform bounds in $p$.

\begin{lemma}\label{lemma:a priori_T}
Given $T>0$ there exists $M=M(T)>0$ such that $\|u_p\|_{L^\infty(Q_T)}\leq M$ for all $p>1$.
\end{lemma}
\begin{proof}
Take $\psi\geq 0$ the unique solution of
\begin{equation*}
\left\{
\begin{array}{l}
\partial_t \psi-\Delta \psi=a \psi \qquad \text{ in } Q_T\\
\psi(0)=u_0, \quad u(t)|_{\partial \Omega}=0
\end{array}
\right.
\end{equation*}
Then
$$
\partial_t \psi-\Delta \psi-a\psi+b(x)\psi^p\geq \partial_t \psi-\Delta \psi-a\psi=0,
$$
hence $\psi$ is a supersolution and from Lemma \ref{lemma_subsuper} we have that $0\leq  u_p\leq \psi$. In particular,
$$
\|u_p\|_{L^\infty(Q_T)}\leq \|\psi\|_{L^\infty(Q_T)}<+\infty, \qquad \text{(as $u_0\in L^\infty(\Omega)$),}
$$
which proves the result.
\end{proof}

\begin{lemma}\label{lemma_aux1}
Given $T>0$, the sequence $\{u_p\}_p$ is bounded in $H^1(0,T;L^2(\Omega))\cap L^\infty(0,T;H^1_0(\Omega))$. Thus there exists $u\in H^1(0,T;L^2(\Omega))\cap L^\infty(0,T;H^1_0(\Omega))$ such that
\[
\begin{split}
& u_p\to u \qquad \text{ strongly in }L^2(Q_T),\text{ weakly in }L^2(0,T;H^1_0(\Omega)),\\
& \partial_t u_p\rightharpoonup \partial_t u \qquad \text{ weakly in }L^2(Q_T).
\end{split}
\]
Moreover, the exists $C=C(T)>0$ such that
\begin{equation}\label{eq:key_estimate}
\iint_{Q_T}b(x)u_p^{p+1}\, dxdt\leq C,\qquad \forall p>1.
\end{equation}
\end{lemma}

\begin{proof}
Multiplying equation \eqref{eq:parabolic} by $u_p$ and integrating in $\Omega$,
\begin{equation*}
\frac{1}{2}\frac{d}{dt}\int_\Omega u_p^2(t)\, dx+\int_\Omega |\nabla u_p(t)|^2\, dx=a\int_\Omega u_p^2(t)\, dx-\int_\Omega b(x) u_p^{p+1}(t)\, dx;
\end{equation*}
integrating now between 0 and $t$,
\begin{multline*}
\frac{1}{2}\int_\Omega u_p^2(t)\, dx + \int_0^t \|\nabla u_p(\xi)\|_2^2\, d\xi+\iint_{Q_t} b(x)u_p^{p+1}\, dxdt  \\
 \leq \frac{1}{2}\int_\Omega u_0^2\, dx+a\iint_{Q_t} u_p^2\, dxdt  \leq\frac{1}{2}\|u_0\|_2^2+at|\Omega| (M(t))^2,
\end{multline*}
and hence
\begin{equation*}
\forall T>0, \quad \{u_p\}_p \text{ is bounded in } L^2(Q_T),\text{ and } \eqref{eq:key_estimate}\text{ holds}.
\end{equation*}
Now using $\partial_t u_p$ as test function ($u_p=0$ on $\partial \Omega$ for all $t>0$, thus $\partial_t u_p(t)\in H^1_0(\Omega)$ for a.e. $t>0$) gives
$$
\int_\Omega (\partial_t u_p)^2\, dx + \frac{1}{2}\frac{d}{dt} \int_\Omega |\nabla u_p(t)|^2\, dx=\frac{a}{2}\frac{d}{dt}\int_\Omega u_p^2(t)\, dx-\frac{d}{dt}\int_\Omega b(x)\frac{u_p^{p+1}(t)}{p+1}\, dx
$$
and again after an integration
\begin{multline}\label{eq:aux4}
\iint_{Q_t} (\partial_t u_p)^2\, dx dt+\frac{1}{2}\int_\Omega |\nabla u_p(t)|^2\, dx+\frac{a}{2}\int_\Omega u_0^2\, dx +\int_\Omega b(x)\frac{u_p^{p+1}(t)}{p+1}\,dx\\
					= \frac{1}{2}\int_\Omega |\nabla u_0|^2\, dx+\int_\Omega b(x)\frac{u_0^{p+1}}{p+1}\, dx + \frac{a}{2}\int_\Omega u_p^2(t)\, dx \leq \frac{1}{2}\|\nabla u_0\|_2^2+\frac{b_\infty|\Omega|}{p+1}+\frac{a}{2}M^2|\Omega|,
\end{multline}
where we have used the fact that $0\leq u_0\leq 1$ whenever $b(x)\neq 0$, together with the previous lemma.

\end{proof}

The proofs of the following two results are inspired by similar computations made in \cite{boccardo, aglio}.

\begin{lemma}\label{lemma:key_lemma}
We have $u(t)\in \K_0$ for a.e. $t>0$.
\end{lemma}
\begin{proof}
Let $\Omega'\Subset \Omega\setminus \Omega_0$ and take $Q_T':=\Omega'\times (0,T)$. Given $m>1$, we will show that $|\{(x,t)\in Q_T': u> m\}|=0$. Denote by \underline{b} the infimum of $b(x)$ over $\overline{\Omega'}$, which is positive by (b2). Recalling \eqref{eq:key_estimate}, we deduce the existence of $C>0$ such that
\[
\begin{split}
0\leq \iint_{\{u_p> m\}\cap Q_T'}\underline{b} u_p\, dxdt&\leq \frac{1}{m^p}\iint_{\{u_p> m\}\cap Q_T'}b(x)u_p^{p+1}\, dxdt\\
											&\leq \frac{1}{m^p}\iint_{Q_T}b(x)u_p^{p+1}\, dxdt\leq \frac{C}{m^p},
\end{split}
\]
hence, as $m>1$ and $\underline{b}>0$,
$$
\lim_{p\to +\infty} \iint_{\{u_p> m\}\cap Q_T'}u_p\, dxdt=0.
$$
Now observe that
\[
\begin{split}
0&=\lim_{p\to +\infty} \iint_{\{u_p> m\}\cap Q_T'}u_p\, dxdt\\
  &=\lim_{p\to+\infty} \left( \int_0^T\int_{\Omega'} u_p\chi_{\{u_p> m\}} \chi_{\{u>m\}}\, dxdt +\int_0^T\int_{\Omega'} u_p \chi_{\{u_p> m\}}\chi_{\{u\leq m\}}\, dxdt\right) \\
				&\geq \lim_{p\to +\infty} \int_0^T \int_{\Omega'} u_p\chi_{\{u_p> m\}} \chi_{\{u> m\}}\, dxdt
\end{split}
\]
As $u_p\chi_{\{u_p>m\}}\chi_{\{u>m\}}\to u \chi_{\{u>m\}}$ a.e. and $|u_p\chi_{\{u_p>m\}}\chi_{\{u>m\}}|\leq L$ on $Q_T$, then by the Lebesgue's dominated convergence Theorem we have
\[
\begin{split}
\lim_{p\to +\infty} \int_0^T \int_{\Omega'} u_p\chi_{\{u_p> m\}} \chi_{\{u> m\}}\, dxdt &=\int_0^T \int_{\Omega'} u \chi_{\{u> m\}}\, dxdt\\
																    &\geq m|\{(t,x)\in Q_T':\ u(t,x)> m\}|\geq 0
\end{split}
\]
and hence $|\{(x,t)\in Q_T':\ u(x,t)> m\}|=0\qquad \text{ whenever }m>1$.
\end{proof}

\begin{lemma}\label{lemma_aux2}
Let $u$ be the limit provided by Lemma \ref{lemma_aux1}. Then, up to a subsequence,
$$
u_p\to u \qquad \text{ strongly in }L^2(0,T;H^1_0(\Omega)).
$$
\end{lemma}
\begin{proof}
Multiplying equation \eqref{eq:parabolic} by $u_p-u$ and integrating in $Q_T$,
\begin{multline*}
\iint_{Q_T} \partial_t u_p (u_p-u)\, dxdt+\iint_{Q_T} \nabla u_p\cdot \nabla (u_p-u)\, dxdt+ \iint_{Q_T} b(x)u_p^p(u_p-u)\, dxdt\\
		=\iint_{Q_T} a u_p (u_p-u)\, dxdt
\end{multline*}
which, after adding and subtracting $\iint_{Q_T} \nabla u\cdot \nabla (u_p-u)\, dxdt$, is equivalent to
\begin{multline*}
\iint_{Q_T} \partial_t u_p (v-u_p)\, dxdt+\iint_{Q_T} |\nabla (u_p-u)|^2\, dxdt+\iint_{Q_T} \nabla u\cdot \nabla (u_p-u)\, dxdt\\
+\iint_{Q_T} b(x)u_p^p(u_p-u)\, dxdt=\iint_{Q_T} a u_p (u_p-u)\, dxdt.
\end{multline*}
By the convergences shown in Lemma \ref{lemma_aux1}, we have that the terms $\iint_{Q_T} \partial_t u_p(u_p-u)\, dxdt$, $\iint_{Q_T} \nabla u\cdot \nabla (u_p-u)\, dxdt$ and $\iint_{Q_T} au_p(u_p-u)\, dxdt$ tend to zero as $p\to +\infty$.
Finally, observe that
\begin{multline*}
\iint_{Q_T} b(x)u_p^p (u_p-u)\, dxdt =\iint_{\{u_p\leq u\}} b(x)u_p^p (u_p-u)\, dxdt\\
			+\iint_{\{u< u_p\}} b(x)u_p^p(u_p-u)\, dxdt \geq \iint_{\{0\leq u_p\leq u\}}b(x)u_p^p(u_p-u)\, dxdt.
\end{multline*}
As $u\leq 1$ a.e. in $Q_T'=(0,T)\times \Omega\setminus \Omega_0$ (cf. Lemma \ref{lemma:key_lemma}), we have
\[
\begin{split}
\left| \iint_{\{0\leq u_p\leq u\}} b(x)u_p^p(u_p-u)\, dxdt\right| &\leq \iint_{\{0\leq u_p\leq u\}\cap Q_T'} b(x)u^p|u_p-u|\, dx dt\\
												&\leq \iint_{Q_T}b_\infty |u_p-u|\, dxdt\to 0,
\end{split}
\]
whence $\liminf \iint_{Q_T} b(x)u_p^p(u_p-u)\, dxdt\geq 0$. Thus
$$
\iint_{Q_T} |\nabla (u_p-u)|^2\, dxdt\to 0 \qquad \text{ as }p\to +\infty,
$$
and the result follows.
\end{proof}

\begin{proof}[Proof of Theorem \ref{thm:3}]
1. The convergences of $u_p$ to $u$ are a consequence of Lemmas \ref{lemma_aux1} and \ref{lemma_aux2}. Let us then prove first of all that
\begin{equation}\label{eq:auxiliary_ineq}
\iint_{Q_T} \partial_t u (v-u)\, dxdt + \iint_{Q_T} \nabla u\cdot \nabla (v-u)\, dxdt \geq \iint_{Q_T} a u(v-u)\, dxdt
\end{equation}
for every $v\in \widetilde \K_0$, where $\widetilde \K_0:=\{v\in L^2(0,T;H^1_0(\Omega)):\ v(t)\in \K_0 \text{ for a.e. } t\in (0,T)\}$.
Fix $v\in \widetilde\K_0$ and take $0<\theta<1$. Multiplying \eqref{eq:parabolic} by $\theta v-u_p$ and integrating we have
\begin{multline*}
\iint_{Q_T} \partial_t u_p(\theta v-u_p)\, dxdt + \iint_{Q_T}\nabla u_p\cdot \nabla (\theta v-u_p)\, dxdt+\iint_{Q_T}b(x)u_p^p(\theta v-u_p)\, dxdt\\
=\iint_{Q_T}au_p(\theta v-u_p)\, dxdt.
\end{multline*}
By Lemmas \ref{lemma_aux1} and \ref{lemma_aux2} we have that
\[
\begin{split}
&\iint_{Q_T}\partial_t u_p (\theta v-u_p)\, dxdt \to \iint_{Q_T}\partial_t u(\theta v-u)\, dxdt,\\
&\iint_{Q_T}\nabla u_p \cdot \nabla (\theta v-u_p)\, dxdt \to \iint_{Q_T} \nabla u\cdot \nabla (\theta v-u)\, dxdt,\\
&\iint_{Q_T} u_p (\theta v-u_p)\, dxdt\to \iint_{Q_T} u (\theta v-u)\, dxdt.
\end{split}
\]
As for the remaining term, as $b(x)=0$ a.e. in $\Omega_0$ and $v\leq 1$ a.e in $\Omega\setminus \Omega_0\times (0,T)$, we have
\begin{multline*}
\iint_{Q_T} b(x)u_p^p(\theta v-u_p)\, dxdt=\iint_{\{0\leq u_p\leq \theta v\}}b(x)u_p^p (\theta v-u_p)\, dxdt\\
					+ \iint_{\{\theta v< u_p\}}b(x)u_p^p (\theta v-u_p)\, dxdt \leq \iint_{Q_T'} b(x)\theta^p|\theta v-u_p|\, dxdt \to 0
\end{multline*}
as $p\to +\infty$, because $\theta<1$. Thus
$$
\iint_{Q_T}\partial_t u (\theta v-u)\, dxdt+\iint_{Q_T}\nabla u\cdot \nabla (\theta v-u)\, dxdt\geq \iint_{Q_T} a u(\theta v-u)\, dxdt
$$
and now we just have to make $\theta\to 1$.

2. Given $v\in \K_0$ and $\xi\in (0,T)$, $h>0$, take
$$
\tilde v(t)=
\left\{
\begin{array}{cc}
v & \text{ if }t\in [\xi, \xi+h],\\
u(t) & \text{ if }t\not\in [\xi, \xi+h].
\end{array}
\right.
$$
Then $\tilde v\in \widetilde\K_0$ and from \eqref{eq:auxiliary_ineq} we have that
$$
\int_\xi^{\xi+h}\int_\Omega \partial_t u(v-u)\, dxdt+\int_\xi^{\xi+h}\int_\Omega \nabla u\cdot \nabla (v-u)\, dxdt\geq \int_\xi^{\xi+h}\int_\Omega a u(v-u)\, dxdt.
$$
Multiplying this inequality by $1/h$ and making $h\to 0$ we get \eqref{eq:limit_problem_2}, as wanted. 

3. Finally, it is easy to show that problem \eqref{eq:limit_problem_2}  has a unique solution. In fact, taking $u_1$ and $u_2$ solutions of \eqref{eq:limit_problem_2} with same initial data, we have
\begin{multline*}
\int_\Omega \partial_t (u_1(t)-u_2(t)) (u_2(t)-u_1(t))+\nabla (u_1(t)-u_2(t))\cdot \nabla (u_2(t)-u_1(t)))\,dx\\
 \geq \int_\Omega a(u_1(t)-u_2(t))(u_2(t)-u_1(t))\, dx,
\end{multline*}
which is equivalent to
$$
\frac{1}{2}\frac{d}{dt}\int_\Omega(u_1(t)-u_2(t))^2\,dx+\int_\Omega |\nabla (u_1(t)-u_2(t))|^2\,dx\leq \int_\Omega a(u_1(t)-u_2(t))^2\,dx.
$$
The fact that $u_1,u_2$ have the same initial data now implies that
$$
\int_\Omega (u_1(t)-u_2(t))^2(t)\,dx\leq e^{2at} \int_\Omega (u_0-u_0)\, dx=0.
$$
Hence $u_p\to u$ for the whole sequence $\{u_p\}_p$, not only for a subsequence.
\end{proof}

\section{The non degenerate case: Proof of Theorem \ref{thm:1}}\label{sec:section b(x)>0}
As stated, the results of the previous section are true even in the case where $\Omega_0= \emptyset$. Let us check that in the non degenerate case (b2') we have a stronger convergence as well as a more detailed characterization for the limit $u$ (cf. \eqref{eq:limit_problem_1}). This is mainly due to the following powerful estimate.

\begin{lemma}\label{lemma:aprioiestimates}
There exists a constant $M>0$ (independent of $p$) such that $\|u_p\|^{p-1}_{L^\infty(Q)}\leq M$ for all $p>1$.
\end{lemma}
\begin{proof}
Let $b_0=\inf_\Omega b>0$ and take $M_p>0$ such that $a M_p-b_0M_p^p=0$. Observe that $as-b_0s^p\leq 0$ for $s\geq M_p$. Take $N_p:=\max\{1,M_p\}$. Multiplying \eqref{eq:parabolic} by $(u_p(t)-N_p)^+$ (recall that $u_p=0$ on $\partial \Omega$, whence $(u_p-N_p)^+=0$ on the boundary as well), we obtain
\begin{multline*}
\frac{1}{2}\frac{d}{dt}\int_\Omega ((u_p-N_p)^+)^2\,dx+\int_\Omega |\nabla (u_p-N_p)^+|^2\,dx=\int_\Omega (au_p-b(x)u_p^p)(u_p-N_p)^+\,dx\\
						\leq \int_\Omega (au_p-b_0 u_p^p)(u_p-N_p)^+\,dx=\int_{\{u_p\geq N_p\}}(au_p-b_0 u_p^p)(u_p-N_p)\,dx\leq 0.
\end{multline*}
Thus
$$
\frac{d}{dt}\int_\Omega ((u_p-N_p)^+)^2\leq 0,\quad \text{ and } \quad \int_\Omega ((u_p(t)-N_p)^+)^2\,dx\leq \int_\Omega ((u_0-N_p)^+)^2\, dx,
$$
 which is zero because $N_p\geq 1$. Then $0\leq u_p(t,x)\leq \max\{1,M_p\}$ and the result now follows from the fact that $M_p=(a/b_0)^{1/(p-1)}$.
\end{proof}

\begin{lemma}\label{lemma:convergence}
For each $t_2>t_1>0$, $q>1$ and $\alpha\in (0,1)$, the sequence $\{u_p\}_p$ is bounded in $W^{2,1}_q(Q_{t_1,t_2})$ and in $C^{1,0}_\alpha(\overline Q_{t_1,t_2})$. Thus 
$$
u_p\rightharpoonup u, \text{ weakly in } W^{2,1}_q(Q_{t_1,t_2}), \qquad u_p \to u \text{ strongly in } C^{1,0}_\alpha(\overline Q_{t_1,t_2}), \ \forall \alpha\in (0,1).
$$
\end{lemma}
\begin{proof}
From Lemma \ref{lemma:aprioiestimates} we get that 
$$\|au_p\|_{L^\infty(Q)}\leq C' M^{1/(p-1)}\leq C'',\quad \text{ and }\quad  \|b(x) u_p^p\|_{L^\infty(Q)}\leq b_\infty M^\frac{p}{p-1}\leq C''',
$$ 
hence 
\begin{equation*}
\|\partial_t u_p-\Delta u_p\|_{L^\infty(Q)} \leq C \qquad \forall p>1 
\end{equation*}
which, together with \cite[IV. Theorems 9.1 $\&$ 10.1]{lady} (see also \cite[Theorem 7.22 $\&$ 7.32]{lieberman}), implies that for every $q>1$, the sequence $\{u_p\}_p$ is bounded in $W^{2,1}_q(Q_{t_1,t_2})$ independently of $p$. Thus we can use the embedding \eqref{eq:embedding} to show that $\{u_p\}_p$ is bounded in $C^{1,0}_\alpha(\overline Q_{t_1,t_2})$. As the embedding $C^{1,0}_\alpha \hookrightarrow \ C^{1,0}_{\alpha'}$ is compact for all $\alpha>\alpha'$, we have the conclusion. 

Observe that by Theorem \ref{thm:3} the whole sequence $u_p$ already converges to $u$ in some spaces, and hence the convergence obtained in this lemma is also for the whole sequence, not only for a subsequence.
\end{proof}

\begin{rem}
It was important to assume $\Omega$ smooth (say $\partial\Omega$ of class $C^2$) to get regularity up to $\partial \Omega$. This will be of crucial importance in the next section. Without such regularity assumption, we would obtain convergence in each set of type $\Omega'\times (t_1,t_2)$ with $\Omega'\Subset \Omega$, $0<t_1<t_2$.
\end{rem}

Now, in view of Theorem \ref{thm:1}, we want to prove that in this case $u$ solves \eqref{eq:limit_problem_1}. By Lemma \ref{lemma:aprioiestimates}, we know that $\|u_p^{p-1}\|_{L^\infty(Q)}\leq M$ for all $p>1$. This implies the existence of $\psi\geq 0$ such that, for every $T>0$,
\begin{align*}
u_p^{p-1}\rightharpoonup \psi \qquad \text{ weak--$\ast$ in $L^{\infty}(Q_T)$, weak in $L^2(Q_T)$.}
\end{align*}
Thus when we make $p\to +\infty$ in \eqref{eq:parabolic} we obtain that the limit $u$ satisfies
$$
\partial_t u-\Delta u=(a-\psi)u.
$$
Moreover, $\|u_p\|_\infty\leq M^\frac{1}{p-1}\to 1$ as $p\to \infty$, which implies, together with Lemma \ref{lemma:convergence}, that $0\leq u\leq 1$. The proof of Theorem \ref{thm:1} will be complete after the following lemmas.

\begin{lemma}\label{lemma_equation}
$\psi=0$ a.e. in the set $\{(t,x)\in Q:\ u(x,t)<1\}$. In particular, this implies that
$$
\partial_t u-\Delta u=a u\chi_{\{u<1\}}\qquad \text{a.e. } (x,t)\in Q.
$$
\end{lemma}
\begin{proof}
Take $(x,t)$ such that $u(x,t)<1$. As $u_p\to u$ in $C^{1,0}_\alpha$ we can take $\delta>0$ such that $u_p\leq 1-\delta$ for large $p$. Then $0\leq u_p^{p-1}\leq (1-\delta)^{p-1}\to 0$ as $p\to +\infty$, whence $\psi(x,t)=0$. Thus $\psi=0$ a.e. on $\{(x,t):\ u(t,x)<1\}$.

Finally, as $u\in W^{2,1}_q$ for every $q\geq 1$, we have $\partial_t u-\Delta u=0$ a.e. on $\{(x,t):\ u(x,t)=1\}$ and the proof is complete.
\end{proof}

\begin{lemma}\label{lemma:u_is_solution_of_varineq}
Let $w$ be a solution of \eqref{eq:limit_problem_1}. Then $w$ solves \eqref{eq:limit_problem_2}.
\end{lemma}
\begin{proof}
Multiply equation \eqref{eq:limit_problem_1} by $v-w$ with $v\in \K$. Then we have
\begin{multline*}
\int_\Omega \partial_t w (v-w)\, dx + \int_\Omega \nabla w\cdot \nabla (v-w)\, dx = a \int_\Omega w \chi_{\{w<1\}}(v-w)\, dx\\
							=a\int_\Omega w(v-w)\, dx-a\int_\Omega (v-1)\, dx \geq a\int_\Omega w (v-w)\, dx,
\end{multline*}
since $v\leq 1$ in $\Omega$.
\end{proof}

\begin{proof}[Proof of Theorem \ref{thm:1}] The convergences $u_p\to u$ strongly in $C^{1,0}_\alpha(\overline Q_{t_1,t_2})$ and weakly in $W^{2,1}_q(Q_{t_1,t_2})$ for every $T>0$ are a consequence of Lemma \ref{lemma:convergence}. By Lemma \ref{lemma_equation}, $u$ satisfies \eqref{eq:limit_problem_1}. Finally Lemma \ref{lemma:u_is_solution_of_varineq} and the uniqueness show for \eqref{eq:limit_problem_2} imply the uniqueness of solution of \eqref{eq:limit_problem_1}.
\end{proof}

\section{Asymptotic behavior as $t\to \infty$. Proof of Theorem \ref{thm:2}}\label{sec:time}

In this section we will study the asymptotic behavior of \eqref{eq:limit_problem_2} as $t\to +\infty$. First we need to understand what happens in the nondegenerate case (b2'), and prove Theorem \ref{thm:2}; for that, as we will see, the convergence up to the boundary proved in Lemma \ref{lemma:convergence} will be crucial. Only afterwards will we be able to prove Theorem \ref{thm:4}.

\subsection{Proof of Theorem \ref{thm:2}}

We start by showing that the time derivative of $u$ vanishes as $t\to +\infty$.

\begin{prop}\label{prop:partial_t_vanishes}
$\|\partial_t u(t)\|_{L^2(\Omega)}\to 0$ as $t\to +\infty$.
\end{prop}
In order to prove this proposition, we will show that $\|\partial u_p(t)\|_{L^2(\Omega)}\to 0$ as $t\to +\infty$, uniformly in $p>1$. To do so, we will use the following result from \cite[Lemma 6.2.1]{Zheng}.
\begin{lemma}\label{lemma:abstract_asymptotic_in_t}
Suppose that $y(t),h(t)$ are nonnegative continuous functions defined on $[0,\infty)$ and satisfy the following conditions:
\begin{align}\label{eq:A1A2A3A4}
y'(t)\leq A_1 y^2+A_2+h(t), \quad \int_0^\infty y(t)\, dt \leq A_3,\ \quad \int_0^\infty h(t)\, dt\leq A_4,
\end{align}
for some constants $A_1,A_2,A_3, A_4>0$. Then
$$
\lim_{t\to +\infty} y(t)=0.
$$
Moreover, this convergence is uniform for all $y$ satisfying \eqref{eq:A1A2A3A4} with the same constants $A_1,A_2,A_3, A_4$. \footnote{This uniformity is not stated in the original lemma, but a close look at the proof allows us to easily obtain that conclusion.}
\end{lemma}
With this in mind, we have:
\begin{lemma}
Let $u_p$ be the solution of \eqref{eq:parabolic} and $a>0$. Then
$$
\|\partial_t u_p(t)\|_2 \to 0 \qquad \text{ as } t \to +\infty, \text{ uniformly in } p>1.
$$
\end{lemma}
\begin{proof}
Let us check that $y(t):=\|\partial_t u_p(t)\|^2_2$ satisfies the assumptions of Lemma \ref{lemma:abstract_asymptotic_in_t}. First of all, \eqref{eq:aux4} implies that 
$$
\int_0^\infty \|\partial_t u_p(t)\|_2^2\, dx \leq \|\nabla u_0\|_2^2+\frac{|\Omega|}{2}+\frac{a}{2}M^2|\Omega|
$$
(recall that in the nondegenerate case $\|u_p\|_{L^\infty(\Omega\times \R^+)}$ is bounded uniformly in $p$, by Lemma \ref{lemma:aprioiestimates}.
Differentiate equation \eqref{eq:parabolic} with respect to $t$:
$$
\partial^2_t u_p-\Delta \partial_t u_p + p u_p^{p-1}\partial_t u_p = a \partial_t u_p,
$$
and multiply it by $\partial_t u_p$ and integrate in $\Omega$, at each time $t$. We obtain
\begin{align*}
\frac{1}{2}\frac{d}{dt}\int_\Omega (\partial_t u_p(t))^2\, dx + & \int_\Omega |\nabla (\partial_t u_p(t))|^2\, dx + p \int_\Omega u_p^{p-1} (t)(\partial_t u_p(t))^2\, dx \\	
							&= a \int_\Omega (\partial_t u_p(t))^2\, dx\leq \frac{a}{2}\left(\int_\Omega (\partial_t u_p(t))^2\, dx\right)^2 + \frac{a}{2}.
\end{align*}
Thus 
$$
\frac{d}{dt}\|\partial_t u_p(t)\|_2^2 \leq a \|\partial_t u_p\|_2^4 + a.
$$
So we can apply the previous lemma with $A_1=a$, $A_2=a$, $A_3=\|\nabla u_0\|_2^2+\frac{|\Omega|}{2}+\frac{a}{2}M^2|\Omega|$, and $h(t)\equiv 0$, $A_4=0$.
\end{proof}

\begin{proof}[Proof of Proposition \ref{prop:partial_t_vanishes}]
From the previous lemma we know that, given $\varepsilon>0$, there exists $\bar t,\ p_0$ such that
$$
\|\partial_t u_p(t)\|^2_2\leq \ep,\qquad \forall t\geq \bar t, \quad \forall p>p_0.
$$
Thus for every $\bar t\leq t_1<t_2$,
$$
\int_{t_1}^{t_2}\|\partial_t u_p(t)\|_2^2\, dt\leq \varepsilon (t_2-t_1) \qquad \forall t\geq \bar t, \quad \forall p>p_0.
$$
As $\partial_t u_p \rightharpoonup \partial_t u$ weakly in $L^2(Q_T)$ for every $T>0$ (cf. Theorem \ref{thm:3}), then taking the $\liminf$ as $p\to +\infty$, we get
$$
\int_{t_1}^{t_2} \|\partial_t u(t)\|_2^2\, dt\leq \varepsilon (t_2-t_1), \qquad \text{ and hence }\|\partial_t u(t)\|^2_2\leq \ep, \qquad \forall t\geq \bar t,
$$
which gives the statement.

\end{proof}

\begin{proof}[Proof of Theorem \ref{thm:2}] Fix $a>\lambda_1(\Omega)$. By taking $v=0$ in \eqref{eq:limit_problem_2} we obtain
$$
\int_\Omega |\nabla u(t)|^2\, dx\leq\int_\Omega (-\partial_t u(t) u(t) +  a u^2)\, dx,
$$
which implies that $\|u(t)\|_{H^1_0(\Omega)}$ is bounded for $t>0$. Therefore, up to a subsequence, we have $u(t)\rightharpoonup \bar u$ in $H^1_0(\Omega)$ as $t\to +\infty$. Given a subsequence $t_n\to +\infty$ such that $u(t_n)\rightharpoonup \bar u$, we know that
$$
\int_\Omega \partial_t u(t_n) (v-u(t_n))\, dx + \int_\Omega \nabla u(t_n)\cdot \nabla (v-u(t_n))\, dx \geq a\int_\Omega u(t_n)(v-u(t_n))\, dx, 
$$
for all $v\in \K$ which, together with Proposition \ref{prop:partial_t_vanishes}, implies that, as $p\to +\infty$,
\begin{equation*}
\int_\Omega \nabla \bar u\cdot \nabla (v-\bar u)\, dx\geq \int_\Omega a \bar u(v-\bar u)\, dx \qquad \forall v\in \K
\end{equation*}
or, equivalently,
$$
-\Delta \bar u = a \bar u \chi_{\{\bar u<1\}}. 
$$
(here we are using the equivalence between these two problems which was shown in \cite{DD} and was stated in the Introduction). Since $\|\bar u\|_\infty\leq 1$ and $a>\lambda_1(\Omega)$, in order to prove that $\bar u=w$ (the unique nontrivial solution of \eqref{eq:obstacle_elliptic_0}) the only thing left to prove is that $\bar u\not\equiv 0$.

2. Let us then check that, for $a>\lambda_1$, $\bar u\not\equiv 0$. Fix any $\bar t>0$. By the maximum principle we have that $u(\bar t,x)>0$ in $\Omega$ and $\partial_\nu u(\bar t,x)<0$ on $\partial \Omega$. By the convergence in $C^{1,0}_\alpha$--spaces \emph{up to the boundary of $\Omega$} (cf. Theorem \ref{thm:1}) we have that, for $p\geq \bar p$, $u_p(\bar t,x)>0$ in $\Omega$ and $\partial_\nu u_p(\bar t,x)<0$ on $\partial \Omega$. Let $\vphi_1$ be the first eigenfunction of the Laplacian in $H^1_0(\Omega)$ with $\vphi_1>0$ and $\|\vphi_1\|_\infty=1.$ Then
\begin{equation}\label{eq:aux1}
c\vphi_1\leq u_p(\bar t,x)\qquad \forall x\in \Omega,\ \forall p\geq \bar p
\end{equation}
for sufficiently small $c$ (independent of $p$). Observe moreover that $\partial_t (c\vphi_1)-\Delta (c\vphi_1)\leq a (c\vphi_1)-b(x)(c\vphi_1)^p$ if and only if 
\begin{equation}\label{eq:aux2}
b(x) c^{p-1}\vphi_1^{p-1}\leq (a-\lambda_1).
\end{equation}
Take $\bar c>0$ such that \eqref{eq:aux1} and \eqref{eq:aux2} hold. Then $\bar c \vphi_1$ is a subsolution of \eqref{eq:parabolic} for sufficiently small $\bar c$, for each $p\geq \bar p$. Then by Lemma \ref{lemma_subsuper} we have that $u_p(t,x)\geq \bar c\vphi_1$ for every $t\geq \bar t$ and $p\geq \bar p$, and hence as $p\to \infty$ also $u(t,x)\geq \bar c\vphi_1(x)$ for every $x\in \Omega$, $t\geq \bar t$.
Thus $\bar u\not\equiv 0$ and $\bar u=w$, the unique solution of \eqref{eq:obstacle_elliptic_0}. From the uniqueness we deduce in particular that $u(t)\rightharpoonup w$ in $H^1_0(\Omega)$ as $t\to \infty$, not only for some subsequence. As for the strong convergence, this is now easy to show since by taking the difference
$$
\partial_t u-\Delta (u(t)-w)=au(t)\chi_{\{u<1\}}-a w \chi_{\{w<1\}}
$$
and multiplying it by $u(t)-w$, we get
$$
\int_\Omega |\nabla (u(t)-w)|^2=-\int_\Omega \partial_t u(t)(u(t)-w)+(au(t)\chi_{\{u<1\}}-a w \chi_{\{w<1\}})(u(t)-w)\to 0
$$
as $t\to \infty$ (recall that both $u(t)$ and $w$ are less than or equal to 1). Thus $u(t)\to w$ strongly in $H^1_0(\Omega)$.

3. The convergence of the coincidence sets follows as in \cite{rodrigues}. As $0\leq \chi_{\{u=1\}}(t)\leq 1$, then there exists a function $0\leq \chi^\ast\leq 1$ such that, up to a subsequence, 
$$
\chi_{\{u=1\}}(t)\rightharpoonup \chi^\ast \qquad \text{ weak-}\ast \text{ in }L^\infty(\Omega), \text{ as } t\to +\infty.
$$ 
Since $\chi_{\{u=1\}} (1-u)=0$ a.e, then also $\chi^\ast (1-w)=0$ a.e. and hence $\chi^\ast=0$ whenever $w<1$. Moreover, from the fact that $\partial_t u-\Delta u=a u (1-\chi_{\{u=1\}})$ a.e. in $Q$ we deduce that $-\Delta w=aw(1-\chi^\ast)$. As $\Delta w=0$ a.e. on $\{w=1\}$ (in fact, $u\in W^{2,q}(\Omega)$ for every $q\geq 1$), we conclude that $\chi^\ast=1$ on $\{w=1\}$, whence $\chi^\ast=\chi_{\{w=1\}}$. Since in general $L^\infty(\Omega)$ weak-$\ast$ convergence of characteristic functions imply strong convergence in $L^q(\Omega)$ for every $q\geq 1$, we have proved \eqref{eq:charact_1}. As a consequence, actually $u(t)\to w$ in $H^2$--norm.

4. For $a<\lambda_1(\Omega)$, the function $0$ attracts all the solutions of \eqref{eq:limit_problem_2} with nonnegative initial data; in fact, by taking $v=0$ in \eqref{eq:limit_problem_2} we obtain
$$
\int_\Omega |\nabla u(t)|^2\, dx \leq a\int_\Omega u(t)^2\, dx-\int_\Omega \partial_t u(t) u(t)\, dx\leq \frac{a}{\lambda_1(\Omega)}\int_\Omega |\nabla u(t)|^2\, dx+\textrm{o}(1)
$$
as $t\to +\infty$, and thus $\|u(t)\|_{H^1_0(\Omega)}\to 0$. 

\end{proof}

%%%
%%Section Asymptotics in t
%%%

\subsection{Proof of Theorem \ref{thm:4}}

Fix $a\in (\lambda_1(\Omega),\lambda_1(\Omega_0))$. In this case we have a stronger result than Lemma \ref{lemma:aprioiestimates}, having a uniform $L^\infty$ bound in $Q=\Omega\times \R^+$.

\begin{lemma}\label{lemma:apriori_final}
For $a\in (\lambda_1(\Omega),\lambda_1(\Omega_0))$, there exists $C>0$ such that $\|u\|_{L^\infty(\R^+\times \Omega)}\leq C$ for all $p>1$.
\end{lemma}
\begin{proof}
Here we follow the line of the proof of Claim 1 in \cite[p. 224]{DDM1}, to which we refer for more details. Define $\Omega_\delta=\{x\in \R^N: \dist(x,\Omega)<\delta\}$. Since $a<\lambda_1(\Omega_0)$, there exists a small $\delta$ such that $a<\lambda_1(\Omega_\delta)$ (by continuity of the map $\Omega\mapsto \lambda_1(\Omega)$). Denoting by $\phi_\delta$ the first eigenfunction of $-\Delta$ in $H^1_0(\Omega_\delta)$ and $\psi$ any extension of $\phi|_{\Omega_{\delta/2}}$to $\overline \Omega$ such that $\min_{\overline \Omega}\psi> 0$, there exists $Q>0$ large enough such that
$$
-\Delta (Q\psi)-aQ\psi+b(x)(Q\psi)\geq 0 \qquad \text{ in } \Omega,
$$
and $u_0\leq Q\psi$ in $\Omega$. Thus $Q\psi$ is a supersolution of \eqref{eq:parabolic} for all $p>1$ and by Lemma \ref{lemma_subsuper} we have
$$
u_p\leq Q\psi\leq M\qquad \text{ for all } (t,x)\in Q.
$$
\end{proof}

\begin{proof}[Proof of Theorem \ref{thm:4}]
1. Fix $a\in (\lambda_1(\Omega),\lambda_1(\Omega_0))$. Having proved Lemma \ref{lemma:apriori_final}, we can repeat the proof of Proposition \ref{prop:partial_t_vanishes} word by word and show that 
$$
\|\partial_t u(t)\|_{L^2(\Omega)}\to 0 \qquad \text{as } t\to +\infty.
$$
By making $v=0$ in \eqref{eq:limit_problem_2}, we obtain once again by Lemma \ref{lemma:apriori_final} that $\|u(t)\|_{H^1_0(\Omega)}$ is bounded for $t>0$. Take $t_n\to +\infty$ such that $u(t_n)\rightharpoonup \bar u$ in $H^1_0(\Omega)$ for some $\bar u\in H^1_0(\Omega)$. Then $\bar u\in \K_0$ and
\begin{equation}\label{eq:final}
\int_\Omega \nabla \bar u\cdot \nabla (v-\bar u)\, dx\geq a\int_\Omega \bar u(v-\bar u)\, dx,\qquad \forall v\in \K_0.
\end{equation}
In \cite{DD}, Dancer and Du have shown that \eqref{eq:final} has a unique nontrivial nonnegative solution $w$. In order to prove that $\bar u=w$ and conclude the proof for this case, we just have to show that $\bar u\not\equiv 0$. This will be a consequence of Theorem \ref{thm:2}. In fact, considering $\phi_p$ as the solution of
\begin{equation*}
\left\{
\begin{array}{l}
\partial_t \phi_p-\Delta \phi_p=a \phi_p -\|b\|_\infty \phi_p^p\qquad \text{ in } Q_T\\
u(0)=v_0, \quad u(t)|_{\partial \Omega}=0,
\end{array}
\right.
\end{equation*}
with $v_0:=\inf\{u_0,1\}$, it is straightforward to see that $\phi_p$ is a subsolution of \eqref{eq:parabolic}, and
$$
u_p\geq \phi_p\to w \qquad \text{ as } p\to +\infty,
$$
where $w\neq 0$ is the unique nontrivial solution of \eqref{eq:obstacle_elliptic_0}. This last statement is a consequence of Theorem \ref{thm:2}, as $0\leq v_0\leq 1$ a.e. in $\Omega$. Thus $\bar u\geq w\not\equiv 0$, which concludes the proof in this case.

2. If $a<\lambda_1(\Omega)$, the same reasoning as in the proof of Theorem \ref{thm:2} yields that $\|u(t)\|_{H^1_0(\Omega)}\to 0$. As for the case $a\geq \lambda_1(\Omega_0)$, if either $\|u(t)\|_\infty$ or $\|u(t)\|_{H^1_0(\Omega)}$ were bounded, it is clear from the proof of Proposition \ref{prop:partial_t_vanishes} that $\|\partial_t u(t)\|_{L^2(\Omega)}\to 0$. Repeating the reasoning of the previous step, we would obtain a nontrivial solution of \eqref{eq:limit_problem_2} for $a\geq \lambda_1(\Omega_0)$, contradicting \cite[Theorem 1.1]{DD}.
\end{proof}

\begin{rem}\label{remark:a_leq_lambda_1}
As for the case $a=\lambda_1(\Omega)$, observe that $c\vphi_1$ is always a steady state solution of \eqref{eq:limit_problem_1} for all $0<c<1$, where $\vphi_1$ denotes the first eigenfunction of $(-\Delta, H^1_0(\Omega))$ with $\|\vphi_1\|_\infty=1$. Hence the long time limit of \eqref{eq:limit_problem_2} in this case will depend on the initial condition $u_0$, and we are only able to conclude that given $t_n\to +\infty$ there exists a subsequence $\{t_{n_k}\}$ such that $u(t_{n_k})$ converges to $c\vphi_1$ for some $c>0$.
\end{rem}

\noindent{\bf Acknowledgements.} The second author would like to thank Jesús Hernández and Pedro Freitas for useful discussions related with this paper.
%% Bibliografia

\end{document}